\documentclass[reqno, 11pt]{amsart}
\usepackage{amsmath}
\usepackage{amsthm}
\usepackage{amssymb}
\usepackage{graphicx}
\usepackage{amsfonts}
\usepackage{latexsym}
\usepackage{hyperref}
\usepackage{cite}

\usepackage[font=small,labelfont=bf]{caption}

\usepackage[usenames,dvipsnames]{xcolor}

\graphicspath{ {Figures-Javad/} }

\let\oldenumerate=\enumerate
	\def\enumerate{
	\oldenumerate
	\setlength{\itemsep}{5pt}
	}
\let\olditemize=\itemize
	\def\itemize{
	\olditemize
	\setlength{\itemsep}{5pt}
	}
	
\usepackage{tikz}
\usetikzlibrary{decorations.markings}
\usetikzlibrary{intersections}
\tikzset{global scale/.style={
    scale=#1,
    every node/.style={scale=#1}
  }
}

\renewcommand{\Im}{\operatorname{Im}}

\newcommand{\T}{\mathbb{T}}
\newcommand{\C}{\mathbb{C}}

\newcommand{\N}{\mathbb{N}}
\newcommand{\Z}{\mathbb{Z}}
\newcommand{\R}{\mathbb{R}}
\newcommand{\D}{\mathbb{D}}

\newcommand{\Aut}{\operatorname{Aut}}

\renewcommand{\phi}{\varphi}

\renewcommand{\leq}{\leqslant}
\renewcommand{\geq}{\geqslant}

\allowdisplaybreaks

\numberwithin{equation}{section}
\theoremstyle{plain}

\newtheorem{Corollary}[equation]{Corollary}
\newtheorem*{Corollary*}{Corollary}
\newtheorem{Theorem}[equation]{Theorem}
\newtheorem*{Theorem*}{Theorem}
\newtheorem{Lemma}[equation]{Lemma}
\theoremstyle{definition}

\allowdisplaybreaks

\begin{document}
\bibliographystyle{amsplain}

\title{Finite Blaschke products: a survey}

\author{Stephan Ramon Garcia}
\address{   Department of Mathematics\\
Pomona College\\
Claremont, California\\
91711 \\ USA}
\email{Stephan.Garcia@pomona.edu}
\thanks{First author partially supported by National Science Foundation Grant DMS-1265973.}

\author{Javad Mashreghi}
\address{D\'epartament de Mathematiques et de Statistique, Universit\'e Laval, Qu\'ebec, QC, G1K 7P4, Canada}
\email{javad.mashreghi@mat.ulaval.ca}
\thanks{Second author partially supported by NSREC}

\author{William T. Ross}
\address{   Department of Mathematics and Computer Science\\
University of Richmond\\
Richmond, Virginia\\
23173 \\ USA}
\email{wross@richmond.edu}

\begin{abstract}
A finite Blaschke product is a product of finitely many automorphisms of the unit disk.
This brief survey covers some of the main topics in the area, including
characterizations of Blaschke products, approximation theorems, derivatives and residues of Blaschke products,
geometric localization of zeros, and selected other topics.
\end{abstract}

\maketitle

\section{Introduction}

A \emph{Blaschke product} is a function of the form
\begin{equation}\label{eq:Bdef}
    B(z) = e^{i\alpha} z^K \prod_{n \geq 1} \frac{|z_{n}|}{z_n} \frac{z_n-z}{1-\overline{z_n} z},
\end{equation}
in which $\alpha \in \R$, $K \in \N_0$, and $\{z_1,z_2,\ldots\}$ is a sequence (finite or infinite) in $\{0 < |z| <1\}$ that 
satisfies the \emph{Blaschke condition}
\begin{equation*}
    \sum_{n \geq 1} (1 - |z_n|) < \infty.
\end{equation*}
This condition ensures that the infinite product \eqref{eq:Bdef} converges uniformly on compact sets of the open unit disk $\D = \{|z| < 1\}$ and thus $B$ defines an 
analytic function on $\D$.  In fact, 
$|B(z)| \leq 1$ on $\D$ and
\begin{equation*}
    \lim_{r \to 1^{-}} B(r \xi)
\end{equation*}
exists and is unimodular for almost every $\xi$ on the unit circle $\T$.  The zeros of $B$ 
are precisely $0$ (if $K \in \N$) along with $z_1,z_2,\ldots$, listed according to multiplicity.
Blaschke products are the zero divisors for many spaces of analytic 
functions on $\D$, the Hardy spaces being the primary examples.  
This subject dates back to the early twentieth century \cite{wB15, MR0012127} and there is an extensive literature devoted to the topic. 

We are concerned here with \emph{finite} Blaschke products
\begin{equation*}
    B(z) = e^{i\alpha} z^K \prod_{k =1}^n \frac{|z_{k}|}{z_k} \frac{z_k-z}{1-\overline{z_k} z},
\end{equation*}
in which $\alpha \in \R$, $K \in \N_0$, and $\{z_1,z_2,\ldots,z_n\}$ is a finite set in $\D \backslash \{0\}$.
Each finite Blaschke product is analytic on a neighborhood of the closed unit disk 
$\D^-$ and is meromorphic on the extended complex plane
$\widehat{\C}:= \C \cup \{\infty\}$.  Since they are products of automorphisms of $\D$, finite Blaschke products enjoy many fascinating
properties.

This brief survey covers some of the main topics in the area, including
characterizations of finite Blaschke products, approximation theorems, derivatives and residues of finite Blaschke products,
geometric localization of zeros, and selected other topics.
We will be selective about which topics and proofs we include. Proofs that enlighten the reader may be included 
while more technical proofs will be omitted.

\section*{Acknowledgement} The authors would like to thank the referee for some useful comments and corrections. 

\section{Automorphisms and Blaschke factors}\label{Section:Factors}

A bijective analytic function $\phi: \D \to \D$ is called an \emph{automorphism} of $\D$.
One can show that $\Aut \D$, the set of all automorphisms of $\D$, is a group under function composition.  Its identity element is
$\operatorname{id}(z) = z$.  For $w \in \D$ and $\gamma \in \T$, let
\begin{equation*}
    \tau_{w} (z) =   \frac{w-z}{1-\overline{w} z} \quad\text{and}\quad
    \rho_{\gamma} (z) =   \gamma z.
\end{equation*}
The following identification of $\Aut \D$ is a consequence of the Schwarz Lemma.

\begin{Theorem} \label{T:automorph-disc}
    If $\phi \in \Aut \D$, then there are unique $w \in \D$ and $\gamma \in \T$ so that $\phi=\rho_\gamma \circ \tau_w$. Thus,
    \begin{equation*}
        \Aut \D  = \{ \rho_\gamma \circ \tau_w : \text{$w \in \D$ and $\gamma \in \T$}\}.
    \end{equation*}
\end{Theorem}

For $z_0 \in \D$, the function
\begin{equation*}
    b_{z_0}(z) = 
    \begin{cases}
        \dfrac{|z_0|}{z_0} \cdot \dfrac{ z_0 - z}{1 - \overline{z_0}z} & \text{if $z_0 \neq 0$},\\[10pt]
        z & \text{if $z_0 = 0$},
    \end{cases}
\end{equation*}
is called a \emph{Blaschke factor}.  These are the automorphisms of $\D$ that are normalized so
that they are nonnegative at the origin:
\begin{equation*}
    b_{z_0} (0) = |z_{0}| \geq 0.
\end{equation*}
In particular, $|b_{z_0}(z)| = 1$ for $z \in \T$ and
\begin{equation*}
b_{z_0}(z) = 
\begin{cases}
    \dfrac{|z_{0}|}{z_0} \cdot\dfrac{z_0-z}{1-\overline{z_0} z} & \text{if $|z| < 1$}, \\[10pt]
      |z_0| & \text{if $z = 0$},\\[10pt]
    \dfrac{1}{\,\overline{b_{z_0}(1/\overline{z})}\,} & \text{if $1 < |z| \leq \infty$},
\end{cases}
\end{equation*}
so $b_{z_0}$ bijectively maps $\D$ to $\D$, $\T$ to $\T$, and $\C\backslash \D$ to $\C \backslash \D$,
respectively.

It is often useful to evaluate the modulus or argument of a Blaschke factor.
For the modulus there is a simple expression:
\begin{equation}\label{E:modulusbz}
    |b_{z_0} (z)|^2 = 1- \frac{(1-|z|^2) 
    (1-|z_0|^2)}{|1-\overline{z_0} z|^2},
\end{equation}
or, equivalently,
\begin{equation}\label{eq:Modb_0Thing}
    \frac{1-|b_{z_0}(z)|^2}{1-|z|^2} = \frac{1-|z_0|^2}{|1-\overline{z_0}  z|^2}.
\end{equation}
The argument requires a little more work.
If $z_0= r_0e^{i\theta_0}$ and $0 < r_0 < 1$, then for $e^{i\theta} \in \T$ we have
\begin{equation*}
    b_{z_0} (e^{i\theta}) 
    = -e^{i(\theta-\theta_0)} \frac{1-r_0
    e^{-i(\theta-\theta_0)}}{1-r_0 e^{i(\theta-\theta_0)}}.
\end{equation*}
Using this identity, one can derive the following.

\begin{Theorem} \label{L:angleont}
    Let $z_0= r_0e^{i\theta_0}$, in which $0 < r_0 < 1$.  Write
    \begin{equation*}
        b_{z_0}(e^{i\theta}) = e^{i \arg b_{z_0}(e^{i\theta}) },
    \end{equation*}
    in which $-\pi \leq \arg b_{z_0}(e^{i\theta}) < \pi$. Then 
    \begin{equation*}
    \arg b_{z_0}(e^{i\theta})
    =
    \begin{cases}
    -\pi & \text{if $\theta = \theta_0$},\\
    -2  \arctan \left(\frac{(1-r_0)}{(1+r_0) \tan(
        \frac{\theta-\theta_0}{2})}\right)
        & \text{if $\theta \in (\theta_0, \theta_0 + 2 \pi)$}.
    \end{cases}
    \end{equation*}
    Here $\arctan$ denotes the principal branch of the arctangent function, the range of which is $(-\pi/2,\pi/2)$.
\end{Theorem}

To compute the argument of a Blaschke factor at points inside $\D$, the situation is more delicate.  
Write
\begin{equation*}
    \frac{b_{z_0}(re^{i\theta})}{|b_{z_0}(re^{i\theta})|} = e^{i\arg b_{z_0} (re^{i\theta})},
\end{equation*}
and observe that the behavior of the function $\theta \mapsto \arg b_{z_0} (re^{i\theta})$ depends 
heavily on $r$.  If $r<|z_0|$, then $b_{z_0}$ has no zeros in the disk $|z| < r$ and hence 
$\arg b_{z_0} (re^{i\theta})$ is a continuous $2\pi$-periodic function on $\R$ (in fact, it is $C^{\infty}$).
If $r > |z_0|$, then the argument principle ensures that $\arg b_{z_0} (re^{i\theta})$ has a jump discontinuity
on any interval of length greater than $2\pi$.   A precise formula for the argument must take this into account.

\begin{Theorem} \label{L:angleonD}
    Let $0 < r_0 < 1$ and $z_0= r_0e^{i\theta_0}$.  Write
    \begin{equation*}
        b_{z_0}(z) = e^{i  \arg b_{z_0}(z) },
    \end{equation*}
    in which $-\pi \leq \arg b_{z_0}(z) < \pi$. Then
    \begin{equation} \label{E:argbzindn0}
        \arg b_{z_0} (z) = \arcsin \frac{\Im(z_0 \overline{z})
        (1-|z_0|^2)}{|z_0| |z_0-z| |1-\overline{z_0} z|}.
    \end{equation}
\end{Theorem}

In many applications, a precise formula is unnecessary.   For instance, the following
corollary was used by Frostman \cite{MR0012127} to discuss the boundary behavior of infinite Blaschke products. 

\begin{Corollary} \label{C:argbzind}
    Let $z_0 \in \D \backslash \{0\}$. Then for every $r \in (0, 1)$,
    \begin{equation*}
        |\arg b_{z_0} (re^{i\theta})| \leq 8\pi 
        \frac{1-|z_0|}{|e^{i\theta}-z_0|}.
    \end{equation*}
\end{Corollary}

\section{Basic properties}
Let $z_1,z_2,\ldots,z_n$ be a finite sequence in $\D$ and let
$\alpha \in \R$. Then
\begin{equation}\label{eq:FBPB}
B(z) = e^{i\alpha} \prod_{k=1}^{n} \frac{z_k-z}{1-\overline{z_k}  z}
\end{equation}
is a finite Blaschke product. Recall that if $P$ and $Q$ are polynomials that have no nonconstant
common factors, then the \emph{order} (or \emph{degree}) of the rational function $f = P/Q$ is 
$\deg f = \max\{\deg P,  \deg Q\}$.  Thus, the finite Blaschke product \eqref{eq:FBPB} has order $n$.

The results of Section \ref{Section:Factors} tell us that
\begin{align}
    |B(z)| &< 1, \quad \text{if $z \in \D$}, \label{E:bles1up1}\\[5pt]
    |B(\zeta)|  &= 1,  \quad  \text{if $\zeta \in \T$}, \label{E:bles1up2}\\[5pt]
    |B(z)| &> 1, \quad \text{if $z \in \widehat{\C} \backslash \D^{-}$}, \label{E:bles1up3}
\end{align}
and
\begin{equation}
    B(z)  = \frac{1}{\,\overline{B(1/\overline{z})}\,}, \quad \text{if $z \in \widehat{\C}$}. \label{E:bles1up4}
\end{equation}
In particular, every finite Blaschke product belongs to $H^{\infty}:=H^\infty(\D)$, the set of bounded analytic functions on $\D$.



If $B$ is a finite Blaschke product \eqref{eq:FBPB} of degree $n$, then $B(z)=0$ has exactly $n$ solutions in $\D$,
counted according to multiplicity.  The solutions are precisely the zeros $z_1,z_2,\ldots,z_n$ of $B$. 
More generally, $B$ is an $n$-to-$1$ function on $\widehat{\C}$.  This follows from the fact that \eqref{eq:FBPB}
permits us to write the equation $B(z) = w$ as a polynomial equation in $z$ of degree $n$.

\begin{Theorem} \label{T:solution-n}
    Let $B$ be a finite Blaschke product of degree $n$. Then for each $w \in \widehat{\C}$ the equation
    $B(z) = w$ has exactly $n$ solutions, counted according to multiplicity.  If $w \in \D$, these solutions belong to $\D$.  
    If $w \in \D_{e}$, these solutions belong to $\D_{e}$.
    If $w \in \T$, these solutions belong to $\T$.  
\end{Theorem}

For $w \in \D$, the equation $B(z)=w$ may have repeated solutions. This occurs, for instance, if $w = 0$, where we obviously have repeated zeros.
However, if $w \in \T$, then repetition does not occur; see Corollary \ref{T:solution-n-11}.

Expressions involving the modulus of a finite Blaschke product are important in many
applications.  For instance, the following theorem, which generalizes
\eqref{eq:Modb_0Thing},
 was used by Frostman to discuss the boundary properties of the derivative of infinite Blaschke products\cite{MR0012127} and by Pekarski{\u\i} \cite{MR652843} to estimate the derivative of a Cauchy transform. 

\begin{Theorem} \label{T:identityforb2}
    Let
    \begin{equation}\label{eq:Bgenc}
        B(z) = \prod_{j=1}^{n} \frac{z-z_j}{1-\overline{z_j}  z},
    \end{equation}
    $B_1=1$, and
    \begin{equation*}\qquad\qquad
        B_k(z) = \prod_{j=1}^{k-1} \frac{z - z_j}{1-\overline{z_j}  z} \quad \text{for $2 \leq k \leq n$}.
    \end{equation*}
    For each $z \in \C \backslash \T$,
    \begin{equation}\label{eq:1BmodId}
        \frac{1-|B(z)|^2}{1-|z|^2} = \sum_{k=1}^{n} |B_k(z)|^2  \frac{1-|z_k|^2}{|1-\overline{z_k}  z|^2}.
    \end{equation}
\end{Theorem}

\begin{proof}
    We induct on $n$.  The case $n=1$ is \eqref{E:modulusbz}.  Suppose that \eqref{eq:1BmodId}
    holds for any Blaschke product of order $n-1$.  By the induction hypothesis,
    \begin{equation} \label{E:t-inductiona-b2}
        \frac{1-|B_n(z)|^2}{1-|z|^2} = \sum_{k=1}^{n-1} |B_k(z)|^2  \frac{1-|z_k|^2}{|1-\overline{z_k}  z|^2}.
    \end{equation}
    Since
    \begin{equation*}
        B(z) = B_n(z)  \frac{z-z_n}{1-\overline{z_n}  z},
    \end{equation*}
    we have
    \begin{align*}
        1-|B(z)|^2 &= 1-|B_n(z)|^2  \left|\frac{z_n-z}{1-\overline{z_n}  z}\right|^2\\
        &= 1-|B_n(z)|^2 +  |B_n(z)|^2 \left( 1- \left|\frac{z_n-z}{1-\overline{z_n}  z}\right|^2 \right)\\
        &= 1-|B_n(z)|^2 +  |B_n(z)|^2  \frac{(1-|z|^2) (1-|z_n|^2)}{|1-\overline{z_n}  z|^2}.
    \end{align*}
    To complete the induction, we appeal to \eqref{E:t-inductiona-b2} and obtain
    \begin{align*}
        \frac{1-|B(z)|^2}{1-|z|^2} &= \frac{1-|B_n(z)|^2}{1-|z|^2} + |B_n(z)|^2  \frac{1-|z_n|^2}{|1-\overline{z_n}  z|^2}\\
        &= \sum_{k=1}^{n} |B_k(z)|^2  \frac{1-|z_k|^2}{|1-\overline{z_k}  z|^2}.\qedhere
    \end{align*}
\end{proof}

The family of finite Blaschke products is conformally invariant; that is, it is 
invariant under the change of variables $z \mapsto \phi(z)$ for all $\phi \in \Aut \D$.
In fact, the degree of a finite Blaschke product is conformally invariant.

\begin{Lemma} \label{L:balschke-conformal}
    Let $B$ be a finite Blaschke product of degree $n$ and let $w \in \D$. 
    Then $\tau_w \circ B$ and $B \circ \tau_w$ are finite Blaschke products of degree $n$.
\end{Lemma}

\begin{proof}
    The function $\tau_w \circ B$ is analytic on $\D$, continuous on $\D^- $, and unimodular on $\T$. 
    Corollary \ref{C:cc-extmfinb} (see below) ensures that $\tau_w \circ B$ is a finite Blaschke product. Moreover,
    $(\tau_w \circ B)(z) = 0$ if and only if $B(z) = w$. Theorem \ref{T:solution-n} tells us that 
    the equation $B(z) = w$ has exactly $n$ solutions in $\D$.  Thus, $\tau_w \circ B$ is a finite Blaschke product of degree $n$. 
    That $B \circ \tau_w$ is a finite Blaschke product of degree can be verified directly.
\end{proof}

The family of all finite Blaschke products is closed under pointwise multiplication. 
It is also closed under composition. In fact, Lemma \ref{L:balschke-conformal} already reveals a special case of this property.

\begin{Theorem} \label{T:t-balschke-composition}
    If $B_1$ and $B_2$ are finite Blaschke products, 
    then $B_1 \circ B_2$ is a finite Blaschke product.  Moreover, if $n_1$ and $n_2$ are the orders of $B_1$ and $B_2$, respectively,
    then the order of $B_1 \circ B_2$ is $n_1n_2$.
\end{Theorem}

\begin{proof}
    Denote the zeros of $B_1$ by $z_1,z_2,\ldots,z_{n_1}$ and write
    \begin{equation*}
        B_1 = \gamma  \tau_{z_1}\tau_{z_2}\cdots\tau_{z_{n_1}},
    \end{equation*}
    in which $\gamma$ is a unimodular constant. Then
    \begin{equation*}
        B_1 \circ B_2 = \gamma  (\tau_{z_1} \circ B_2) \cdot (\tau_{z_2} \circ B_2)  \cdots  (\tau_{z_{n_1}} \circ B_2).
    \end{equation*}
    By Lemma \ref{L:balschke-conformal}, each $\tau_{z_k} \circ B_2$ is a finite Blaschke product of order $n_2$. 
    Consequently, $B_1 \circ B_2$ is a finite Blaschke product of order $n_1n_2$.
\end{proof}

\section{Characterizations of finite Blaschke products}

There are many function-theoretic characterizations of Blaschke products.  We 
recall here just a few of them.

\begin{Theorem}[Fatou \cite{MR1504825}] \label{T:extmfinb}
    If $f$ is analytic on $\D$ and 
    \begin{equation*}
    \lim_{|z| \to 1} |f(z)| = 1,
    \end{equation*}
    then $f$ is a finite Blaschke product.
\end{Theorem}

\begin{proof}
    Since $|f(z)| \to 1$ uniformly as $|z| \to 1$, there is an $r < 1$ so that $f$ is nonvanishing on the annulus
    $\{r \leq |z| < 1\}$.  Thus, $f$ has at most a finite number of zeros in $\D$. 
    Let $B$ be the finite Blaschke product formed from the zeros of $f$, repeated according to multiplicity.
    Then $f/B$ and $B/f$ are analytic in $\D$ and their moduli tend uniformly to $1$ as we approach $\T$. 
    The Maximum Modulus 
    Principle ensures that $|f/B| \leq 1$ and $|B/f| \leq 1$ on $\D$, so
    $f/B$ is constant on $\D$.  Since this constant must be unimodular,  $f$ is a unimodular scalar
    multiple of $B$.
\end{proof}

Each finite Blaschke product belongs the disk algebra $\mathcal{A}(\D)$,
the set of analytic functions on $\D$ that extend continuously on $\D^-$.  In fact, 
the finite Blaschke products are the \emph{only} elements of $\mathcal{A}(\D)$
that map $\T$ into $\T$.

\begin{Corollary} \label{C:cc-extmfinb}
    If $f \in \mathcal{A}(\D)$ is unimodular on $\T$, then $f$ is a finite Blaschke product.
\end{Corollary}

There is also meromorphic version of this. 

\begin{Corollary}\label{aoirhfc}
    Suppose $f$ is meromorphic on $\D$ and extends continuously to $\T$.
    If $f$ is unimodular on $\T$, then $f$ is a quotient of two Blaschke products. 
\end{Corollary}

For $z_1,z_2,\ldots,z_n$ in $\D$, we may write
\begin{equation}\label{eq:QuotientOfThings}
    B(z) =  \prod_{k=1}^{n} \frac{z-z_k}{1-\overline{z_k}  z},
     = \frac{\alpha_0+\alpha_1 z + \cdots + \alpha_n z^n}{\overline{\alpha}_n+\overline{\alpha}_{n-1}z+\cdots+\overline{\alpha}_0z^n},
\end{equation}
in which $\alpha_0,\alpha_1,\ldots,\alpha_{n-1} \in \C$ and $\alpha_n =1$. 
The roots of the polynomial in the numerator are precisely $z_1,z_2,\ldots,z_n$.
Conversely, any rational function of this form is a Blaschke product of order $n$.

\begin{Corollary} \label{C:extmfinb-rational}
    Let $\alpha_0,\alpha_1,\ldots,\alpha_n \in \C$ and $\alpha_n \neq 0$.  If all the roots of 
    \begin{equation*}
        P(z) = \alpha_0+\alpha_1 z + \cdots + \alpha_n z^n
    \end{equation*}
    are in $\D$, then
    \begin{equation*}
        f(z) = \frac{P(z)}{ z^n  \overline{P(1/\overline{z})} }
        = \frac{\alpha_0+\alpha_1 z + \cdots + \alpha_n z^n}{\overline{\alpha}_n+\overline{\alpha}_{n-1}z+\cdots+\overline{\alpha}_0z^n}
    \end{equation*}
    is a finite Blaschke product of order $n$.
\end{Corollary}

\begin{proof}
    Observe that $|f| = 1$ on $\T$ and apply Corollary \ref{C:cc-extmfinb}.
We can be more specific: if the roots of $P(z)$ are $z_1,z_2,\ldots,z_n \in \D$, then
$f$ is a unimodular scalar multiple of the finite Blaschke product with zeros $z_1,z_2,\ldots,z_n$.
\end{proof}

Theorem \ref{T:solution-n} says that if $B$ is a finite Blaschke product of order $n$, then the restricted function $B : \D \to \D$ has 
constant valence $n$ for each $w \in \D$. This property actually characterizes the finite Blaschke products of order $n$.

\begin{Theorem}[Fatou \cite{MR1504787, MR1504797, MR1504792}] \label{T:fatou-valence}
    Let $f : \D \to \D$ be a surjective analytic function of constant valence $n \geq 1$. Then $f$ is a finite Blaschke product of order $n$.
\end{Theorem}

\begin{proof}
    We follow the proof Rad\'o \cite{rad0-1922} and show that
    \begin{equation}\label{eq:RadoLimit}
        \lim_{|z| \to 1} |f(z)| = 1.
    \end{equation}
    If we can do this, then Theorem \ref{T:extmfinb} will ensure that $f$ is a finite Blaschke product.
    Suppose towards a contradiction that \eqref{eq:RadoLimit} fails.  Then there is a sequence $z_m$ of distinct points in $\D$ and a $w_0 \in \D$ so that 
    \begin{equation}
        \lim_{m \to \infty} |z_m| = 1 \quad \text{and} \quad \lim_{m \to \infty} f(z_m) = w_0.
    \end{equation}
    By hypothesis, $f(z_m) \neq w_0$ for all but finitely many $m$. 
    Let $a_1,a_2,\ldots,a_k$ be the distinct solutions of $f(z)=w_0$, respectively, with multiplicities $n_1,n_2,\ldots,n_k$. 
    By assumption, we have
    \begin{equation*}
        n_1+n_2+\cdots+n_k  = n.
    \end{equation*}
    About each point $a_j$, the function $f$ has a power series expansion
    \begin{equation*}
        f(z) = w_0 + \sum_{k = n_j}^{\infty} \frac{f^{(k)}(a_j)}{k!}  (z-a_j)^k,
    \end{equation*}
    in which $f^{(n_j)}(a_j) \neq 0$. 
    If $\epsilon_j > 0$ is sufficiently small, we can write
    \begin{equation} \label{E:t-fgjtomj}
        f(z) = w_0 + \big((z-a_j) f_j(z)\big)^{n_j}
    \end{equation}
    for $z$ in $$D(a_j,\epsilon_j) = \{ z : |z - a_j| < \epsilon_j\},$$ in which 
    $f_j$ is nonvanishing on $D(a_j, \epsilon_j)$. 
    Without loss of generality, we impose the extra restrictions
    \begin{equation*}
        \epsilon_j < (1-|a_j|)/2 \quad\text{and}\quad
        \epsilon_j < \min \{ |a_j-a_i|/2 : 1 \leq i \leq k, i \neq j \}
    \end{equation*}
    to ensure that the $D(a_j, \epsilon_j)$ are pairwise disjoint and do not intersect $\T$.
    
    Since $g_j(z) = (z-a_j)  f_j(z)$ has a simple zero at $a_j$, it is injective on a small neighborhood of $a_j$. 
    Thus, if necessary, we can make $\epsilon_j$ even smaller so that $g_j(z)$ is injective on $D(a_j, \epsilon_j)$. 
    The Open Mapping Theorem ensures that
    \begin{equation*}
        \bigcap_{j=1}^{k} \left(g_j(D(a_j, \epsilon_j))\right)^{n_j}
    \end{equation*}
    is an open set that contains the origin.  Let $\epsilon>0$ be small enough so that
    \begin{equation*}
        D(0, \epsilon) \subset \bigcap_{j=1}^{k} \Big(g_j\big(D(a_j, \epsilon_j) \big) \Big)^{n_j},
    \end{equation*}
    and set
    \begin{equation*}
        \qquad V_j = g_j^{-1}\big( D(0, \epsilon^{1/n_j} ) \big) \subset D(a_j, \epsilon_j), \quad \text{for $1 \leq j \leq k$}.
    \end{equation*}
    Observe that $g_j : V_j \to D(0,\epsilon^{1/n_j})$ is bijective with $g_j(a_j)=0$ and that 
    the open sets $V_j$ are pairwise disjoint and do not intersect $\T$.  Consequently, \eqref{E:t-fgjtomj} tells us that
    for each  $w \in D(w_0,\epsilon) \backslash \{w_0\}$, the equation $f(z)=w$ has exactly $n_j$ distinct solutions in $V_j$, for each $j$.
    
    Since $f(z_m) \to w_0$ and $|z_m| \to 1$, for sufficiently large $m$ we have $f(z_m) \in D(w_0,\epsilon)$, $f(z_m) \neq w_0$, and 
    \begin{equation*}
        z_m \not \in \bigcup_{j=1}^{k} V_j.
    \end{equation*}
    Fix any such $m$, and set $w_m = f(z_m)$.  Then each $V_j$ contains $n_j$ distinct points $z$ so that $f(z) = w_m$. 
    Thus, $f(z) = w_m$ has at least
    \begin{equation*}
        n_1+n_2+\cdots+n_k +1 = n+1
    \end{equation*}
    solutions.  This is a contradiction.
\end{proof}

\section{Approximation theorems}

If $f$ is analytic on $\D$ and can be uniformly approximated on $\D$ by a sequence of finite Blaschke products, 
then it is uniformly continuous on $\D$.  Consequently, $f$ has a unique continuous extension to $\D^-$ and, moreover,
this extension is unimodular on $\T$.  So $f$ is itself a finite Blaschke product by  
Corollary \ref{C:cc-extmfinb}.  Consequently, the finite Blaschke products form a proper closed subset of 
\begin{equation*}
    \mathcal{B}_{H^\infty} = \big\{f \in H^{\infty}(\D): \sup_{z \in \D} |f(z)| \leqslant 1\big\},
\end{equation*}
the closed unit ball of $H^{\infty}$.

Since a generic element of $\mathcal{B}_{H^\infty}$ need not be continuous on $\T$,
we cannot expect uniform approximation by finite Blaschke products.  On the other hand, 
if we endow $H^\infty$ with the topology of uniform convergence on compact sets, which is coarser than the norm topology,
then the finite Blaschke products are dense in $\mathcal{B}_{H^\infty}$.

\begin{Theorem}[Carath\'eodory \cite{MR0064861}] \label{T:caratheodory-fbp}
    For each $f \in \mathcal{B}_{H^\infty}$ there is a sequence of finite Blaschke products that 
    converges uniformly on compact subsets of $\D$ to $f$.
\end{Theorem}

\begin{proof}
    It suffices to show that for each $f \in \mathcal{B}_{H^{\infty}}$ and $n \geq 1$, there is a 
    finite Blaschke product $B_n$ so that $f- B_n$ has a zero of order at least $n$ at the origin.
    If we can show this, then the Schwarz Lemma will ensure that
    \begin{equation*}
        |f(z) - B_n(z)| \leq 2 |z|^{n+1}
    \end{equation*}
    on $\D$, so $B_n \to f$ uniformly on compact subsets of $\D$.  
    
    We proceed by induction on $n$.
    For each $f \in \mathcal{B}_{H^\infty}$, we have $c_0 = f(0) \in \D^-$. 
    If $|c_0|=1$, then the Maximum Principle ensures that $f$ is a unimodular constant and there is nothing to prove.
    If $|c_0| < 1$, then $B_0(z) = -\tau_{-c_0}(z)$ is a finite Blaschke product so that $f - B_0$ vanishes at the origin.

    Suppose, for our induction hypothesis, that for each $f \in \mathcal{B}_{H^\infty}$ there
    is a finite Blaschke product $B_{n-1}$ so that $f - B_{n - 1}$ has a zero of order at least $n$ at the origin.
    Given $f \in \mathcal{B}_{H^\infty}$, the Schwarz Lemma ensures that
    \begin{equation*}
        g(z) = \frac{\tau_{c_0}(f(z))}{z}
    \end{equation*}
    belongs to $\mathcal{B}_{H^\infty}$.  Hence there is a finite Blaschke product $B_{n-1}$ 
    so that $g-B_{n-1}$ has a zero of of order at least $n$ at the origin. Since
    \begin{equation*}
        f(z) = \tau_{c_0}(z g(z)), 
    \end{equation*}
    and since
    \begin{equation*}
        B_n(z) = \tau_{c_0}(zB_{n-1}(z)), 
    \end{equation*}
    is a finite Blaschke product by Lemma \ref{L:balschke-conformal}, we expect that $B_n$ has the desired properties. To establish this, it is enough to observe that
    \begin{equation*}
        \tau_{c_0}(z_2) - \tau_{c_0}(z_1) = \frac{(1-|c_0|^2)  (z_1-z_2)}{(1-\overline{c}_0 z_1)  (1-\overline{c}_0 z_2)}
    \end{equation*}
    for $z_1,z_2\in \D$.
    Because of the factor $z(g(z)-B_{n-1}(z))$, the difference
    \begin{equation*}
        f(z) - B_n(z) = \tau_{c_0}\big( zg(z) ) - \tau_{c_0}(zB_{n-1}(z) \big)
    \end{equation*}
    is divisible by $z^{n+1}$.
\end{proof}

Let us return to the norm topology. Since each finite Blaschke product belongs to $\mathcal{B}_{\mathcal{A}(\D)}$, the closed unit ball of the disk algebra $\mathcal{A}(\D)$, 
any convex combination of finite Blaschke products also belongs to $\mathcal{B}_{\mathcal{A}(\D)}$.  In fact,
the convex hull of the set of all finite Blaschke products is dense in $\mathcal{B}_{\mathcal{A}(\D)}$.

\begin{Theorem}[Fisher \cite{MR0233995}] \label{T:fisher}
    Each $f \in \mathcal{B}_{\mathcal{A}(\D)}$ can be uniformly approximated on $\D^- $ by a sequence of convex combinations of finite Blaschke products.
\end{Theorem}

\begin{proof}
Fix $f \in \mathcal{B}_{\mathcal{A}(\D)}$ and $\epsilon > 0$.
Since $f$ is continuous on $\D^- $, there is a $t \in [0,1)$ so that
\begin{equation} \label{E:tfttofunif}
    \|f_t-f\|_\infty < \frac{\epsilon}{2},
\end{equation}
in which $f_t(z) = f(tz)$ for $z \in \D$. 
Theorem \ref{T:caratheodory-fbp} provides a finite Blaschke product $B$ so that
\begin{equation*}
\|f_t - B_t \|_\infty < \frac{\epsilon}{2}.
\end{equation*}
Thus we have a finite Blaschke product $B$ such that
\begin{equation*}
\|f - B_t \|_\infty < \epsilon.
\end{equation*}
If we can show that $B_t$ is itself a convex combination of finite Blaschke products, the proof will be complete.

Since $(gh)_t = g_t h_t$ for any analytic functions $g$ and $h$ on $\D$, 
the family of convex combinations of finite Blaschke products is closed under multiplication. 
Hence we just need to focus our efforts on 
\begin{equation*}
    B(z) = \frac{\alpha-z}{1-\overline{\alpha}  z}.
\end{equation*}
The combination on the right-hand side of the following
\begin{equation} \label{E:tbtcombination}
    B_t(z) = \frac{t(1-|\alpha|^2)}{1-|\alpha|^2  t^2} 
    \cdot \frac{\alpha t -z}{1-\overline{\alpha}t  z} + \frac{|\alpha|  (1-t^2)}{1-|\alpha|^2  t^2} \cdot e^{i\arg \alpha}
\end{equation}
is almost what we want; it is a linear combination, with positive coefficients, of a Blaschke factor and a unimodular constant.  
However, the coefficients need not sum to one. Fortunately,
\begin{equation*}
1-\frac{t(1-|\alpha|^2)}{1-|\alpha|^2 t^2} - \frac{|\alpha| (1-t^2)}{1-|\alpha|^2 t^2} = \frac{(1-t)(1-|\alpha|)}{1+|\alpha|t}> 0,
\end{equation*}
so we add the combination
\begin{equation*}
    0 = \frac{(1-t)(1-|\alpha|)}{2(1+|\alpha|t)} \cdot 1 + \frac{(1-t)(1-|\alpha|)}{2(1+|\alpha|t)} \cdot (-1)
\end{equation*}
to both sides of \eqref{E:tbtcombination} to obtain
\begin{align*}
B_t(z) &= \frac{t(1-|\alpha|^2)}{1-|\alpha|^2 t^2} \cdot \frac{\alpha t -z}{1-\overline{\alpha}t z} + \frac{|\alpha| (1-t^2)}{1-|\alpha|^2 t^2} \cdot e^{i\arg \alpha}\\
&+ \frac{(1-t)(1-|\alpha|)}{2(1+|\alpha|t)} \cdot 1 + \frac{(1-t)(1-|\alpha|)}{2(1+|\alpha|t)} \cdot (-1),
\end{align*}
which is a convex combination of four finite Blaschke products (three of them are unimodular constants). 
\end{proof}

If $B_1$ and $B_2$ are finite Blaschke products, then $B_1/B_2$ is a
continuous unimodular function on $\T$.   The family of all such quotients is uniformly dense in the
set of continuous unimodular functions.
A weaker version of the following result (Theorem \ref{T:helson-sarason} below) is Corollary \ref{aoirhfc}, in which $u$ is the unimodular boundary function for 
a meromorphic function on $\D$ with a finite number of zeros and poles. 

\begin{Theorem}[Helson--Sarason \cite{MR0236989}] \label{T:helson-sarason}
    Let $u$ be any continuous unimodular function on the unit
    circle $\T$ and let $\epsilon>0$. Then there are finite
    Blaschke products $B_1$ and $B_2$ such that
    \begin{equation*}
    \bigg\| u - \frac{B_1}{B_2}\bigg\|_{L^{\infty}(\T)} <
    \epsilon.
    \end{equation*}
\end{Theorem}

In the preceding theorems, the approximating Blaschke products may have repeated zeros.  This restriction can be overcome. 

\begin{Theorem} \label{T:approx-famil-distinct}
    Let $B$ be a finite Blaschke product of order $n$. 
    Then there is a family of finite Blaschke products $\{B_\epsilon$: $0<\epsilon<\epsilon_0\}$ with the following properties:
    \begin{enumerate}
        \item each $B_\epsilon$ is of order $n$;
        \item each $B_\epsilon$ has distinct zeros;
        \item for each $\epsilon$, $B_\epsilon(0) \neq 0$ and $B_\epsilon'(0) \neq 0$;
        \item as $\epsilon \to 0$, $B_\epsilon$ converges uniformly to $B$ 
        on any compact subset of $\C$ that does not contain a pole of $B$. In particular, $B_\epsilon$ converges uniformly to $B$ on $\D^- $.
    \end{enumerate}
\end{Theorem}


\section{A generalized Rouch\'e theorem and its converse}
A classical theorem of Rouch\'e says that if $f$ and $g$ are analytic inside and on a closed rectifiable curve $\Gamma$ and if $|f+g|<|g|$ on $\Gamma$, 
then $f$ and $g$ have the same number of zeros inside $\Gamma$.  A stronger version is:

\begin{Theorem}[Glicksberg \cite{MR0396919}] \label{T:roche}
    If $f$ and $g$ are analytic inside and on a simple closed rectifiable curve $\Gamma$ and if
    \begin{equation*}
        \text{$|f+g|<|f|+|g|$ on $\Gamma$},
    \end{equation*}
    then $f$ and $g$ have the same number of zeros inside $\Gamma$.
\end{Theorem}

Suppose that $f$ and $g$ are analytic on $|z| < R$ for some $R>1$
and that they have no zeros on $\T$. Let $B_1$ and $B_2$ be finite
Blaschke products of the \emph{same} order. If
\begin{equation*}
|B_1f+B_2g|<|f|+|g|
\end{equation*}
on $\T$, then Theorem \ref{T:roche} ensures that $f$
and $g$ have the same number of zeros in $\D$. The
converse of this result is also true.

\begin{Theorem}[Challener--Rubel \cite{MR653504}]
    Suppose that $f$ and $g$ are analytic on $|z| < R$ for some $R>1$ and that they have no zeros on $\T$. If $f$ and $g$
    have the same number of zeros in $\D$, then there
    are finite Blaschke products $B_1$ and $B_2$ of the same order such that
    \begin{equation*}
        \text{$|B_1f+B_2g|<|f|+|g|$ on $\T$.}
    \end{equation*}
\end{Theorem}

\begin{proof}
    Since $f$ and $g$ are continuous and nonzero on $\T$,
    \begin{equation*}
        m = \min_{\zeta \in \T} \big\{ |f(\zeta)|, |g(\zeta)| \big\} >0
    \end{equation*}
    and
    \begin{equation*}
        M = \max_{\zeta \in \T} \big\{ |f(\zeta)|, |g(\zeta)| \big\} < \infty.
    \end{equation*}
    Let $h=g/f$ and $u = h/|h|$. Since $-u$ is a continuous
    unimodular function on $\T$, Theorem \ref{T:helson-sarason} provides 
    two finite Blaschke products $B_1$ and $B_2$ so that
    \begin{equation*}
        \bigg\| u + \frac{B_1}{B_2}\bigg\|_{L^{\infty}(\T)} < \frac{m}{M}.
    \end{equation*}
    On $\T$, we have
    \begin{align*}
        |B_1f+B_2g| 
        &= |f| \bigg| \frac{g}{f} + \frac{B_1}{B_2}\bigg|
        = |f| \bigg| h + \frac{B_1}{B_2}\bigg|\\
        &= |f| \bigg| u |h| + \frac{B_1}{B_2}\bigg| 
        = |f| \bigg| u (|h|-1) + \Big(u+\frac{B_1}{B_2}\Big)\bigg|\\
        &\leq \big| |f|-|g|\big| + m
         < |f| + |g|.
    \end{align*}
    Since $|f| = |B_1f|$ and $|g| = |B_2g|$ on $\T$, 
    Rouch\'e's Theorem ensures that $B_1f$ and $B_2g$ have the same number of zeros
    inside $\D$. Since $f$ and $g$ have the same
    number of zeros inside $\D$, the Blaschke products $B_1$ and $B_2$ must have the same order.
\end{proof}


\section{The derivative of a finite Blaschke product}

If
\begin{equation*}
    B(z) = \prod_{k=1}^{n} \frac{z_k-z}{1-\overline{z_k}  z}
\end{equation*}
and 
\begin{equation*}
    \qquad B_j(z) = \prod_{\substack{k=1\\k \neq j}}^{n} \frac{z_k-z}{1-\overline{z_k}  z}, \quad \text{for $1 \leq j \leq n$},
\end{equation*}
then a computation confirms that
\begin{equation} \label{E:estb'1-0-1-fin}
    B'(z) = -\sum_{k=1}^{n} \frac{1-|z_k|^2}{(1-\overline{z_k}  z)^2} B_k(z).
\end{equation}
In particular,  
\begin{equation*}
    \qquad B'(z_j) = -  \frac{1}{1-|z_j|^2} \prod_{\substack{k=1\\k \neq j}}^{n} \frac{z_k-z_j}{1-\overline{z_k}  z_j}, \quad \text{for $1 \leq j \leq n$}.
\end{equation*}
Divide both sides of \eqref{E:estb'1-0-1-fin} by $B$ to obtain the logarithmic derivative 
\begin{equation} \label{E:estb'1-0-fin}
    \frac{B'(z)}{B(z)} = \sum_{k=1}^{n}
    \frac{1-|z_k|^2}{(1-\overline{z_k}  z)(z-z_k)}
\end{equation}
of $B$; this is valid on $\C \backslash \big\{ z_k, 1/\overline{z_k}: 1 \leq k \leq n \big\}$. 
At each $e^{i\theta} \in \T$, we have
\begin{equation} \label{E:estb'1-0-2-fin}
    \frac{B'(e^{i\theta})}{e^{-i\theta} B(e^{i\theta})} = \sum_{k=1}^{n} \frac{1-|z_k|^2}{|e^{i\theta}-z_k|^2},
\end{equation}
and hence
\begin{equation} \label{E:estb'1-0-3-fin}
    |B'(e^{i\theta})| = \sum_{k=1}^{n} \frac{1-|z_k|^2}{|e^{i\theta}-z_k|^2}.
\end{equation}
The following result is a direct consequence of \eqref{E:estb'1-0-3-fin}.

\begin{Lemma} \label{L:nozerob'ont}
    If $B$ is a finite Blaschke product, then
    $B'(e^{i\theta}) \neq 0$ for all $e^{i\theta} \in \T$.
\end{Lemma}

In the light of the preceding lemma, we can slightly strengthen Theorem \ref{T:solution-n}
(it can also be deduced from Theorem \ref{L:angleont}).

\begin{Corollary} \label{T:solution-n-11}
    If $B$ is a finite Blaschke product of degree $n$, then for each $w \in \T$, the equation
    $B(z) = w$ has exactly $n$ distinct solutions on $\T$.
\end{Corollary}

Since $B'$ is continuous on $\D^- $, we have
\begin{equation*}
    \lim_{z \to e^{i\theta}} |B'(z)| = |B'(e^{i\theta})|,
\end{equation*}
in which convergence is uniform with respect to $\theta$.  A more subtle expression for $|B'(e^{i\theta})|$
is provided by the following theorem.

\begin{Theorem} \label{T:heins-b'to1}
    If $B$ is a finite Blaschke product, then for all $e^{i\theta} \in \T$,
    \begin{equation*}
        \lim_{z \to e^{i\theta}} \frac{1-|B(z)|^2}{1-|z|^2} = \lim_{z \to e^{i\theta}} \frac{1-|B(z)|}{1-|z|} = |B'(e^{i\theta})|,
    \end{equation*}
    in which convergence is uniform with respect to $\theta$.
\end{Theorem}

\begin{proof}
    If $B$ denotes the finite Blaschke product \eqref{eq:Bgenc},
    then \eqref{eq:1BmodId} tells us that
    \begin{equation*}
        \lim_{z \to e^{i\theta}} \frac{1-|B(z)|^2}{1-|z|^2}  = \sum_{k=1}^{n}  \frac{1-|z_k|^2}{|1-\overline{z_k} e^{i\theta}|^2}
        = \sum_{k=1}^{n} \frac{1-|z_k|^2}{|e^{i\theta}-z_k|^2} = |B'(e^{i\theta})|;
    \end{equation*}
    the final equality is \eqref{E:estb'1-0-3-fin}.  Since 
    the modulus of a finite Blaschke product tends uniformly to $1$ as one approaches the boundary $\T$, \eqref{eq:1BmodId} shows that 
     convergence is uniform in $\theta$.  The identity
    \begin{equation*}
        \lim_{z \to e^{i\theta}} \frac{1-|B(z)|^2}{1-|z|^2} 
        = \lim_{z \to e^{i\theta}} \frac{1-|B(z)|}{1-|z|}\cdot\frac{1+|B(z)|}{1+|z|}
        = \lim_{z \to e^{i\theta}} \frac{1-|B(z)|}{1-|z|}
    \end{equation*}
    follows for the same reason.
\end{proof}

As it turns out, the previous theorem holds in a more general setting under the heading of an ``angular derivative''. 

\begin{Corollary} \label{T:heins-b'to1-ont}
    Let $B$ be a finite Blaschke product. Then
    \begin{equation*}
        \lim_{|z| \to 1} \frac{|B'(z)| (1-|z|^2)}{1-|B(z)|^2} = 1.
    \end{equation*}
\end{Corollary}


\section{The zeros of $B'$}

If $B$ is a finite Blaschke product of degree $n$, then \eqref{eq:QuotientOfThings}
and the quotient rule for derivatives imply that $B' = P/Q$, in which $P$ and $Q$ are polynomials and $\deg P \leq 2n-2$. 
Lemma \ref{L:nozerob'ont} ensures that there are no zeros of $B'$ on $\T$; they are either in $\D$ or in $\D_e:= \widehat{\C}\backslash \D^-$.  We can be more specific.

\begin{Lemma} \label{L:zerosindde}
    Let $B$ be finite Blaschke product. 
    Suppose that $w \in \C \backslash \{0\}$ and that $B(w) \neq 0$. Then $B'(w)=0$ if and only if $B'( 1/\overline{w} ) = 0$.
\end{Lemma}

\begin{proof}
    For each $z \in \C \backslash \{0\}$, \eqref{E:bles1up4} tells us that
    \begin{equation*}
        B(z) \overline{B(1/\overline{z})} = 1.
    \end{equation*}
    Taking the derivative with respect to $z$ reveals that
    \begin{equation*}
        B'(z)\overline{B(1/\overline{z})} - \frac{1}{z^2} B(z)\overline{B'(1/\overline{z})} = 0,
    \end{equation*}
    so
    \begin{equation*}
        B'(w)=0 \quad \iff\quad B'( 1/\overline{w} ) = 0. \qedhere
    \end{equation*}
\end{proof}

\begin{Theorem} \label{T:zeros-b'-indde}
    Let $B$ be a finite Blaschke product of order $n$. Write
    \begin{equation*}
        B(z) = e^{i\beta} z^{j_0} \prod_{k=1}^{m} \bigg(
        \frac{z_k-z}{1-\overline{z_k}  z}  \bigg)^{j_k},
    \end{equation*}
    in which $\beta \in \R$, $j_k$ are positive integers with
    \begin{equation*}
        j_0+j_1+\cdots+j_m = n,
    \end{equation*}
    and $z_1,z_2,\ldots,z_m$ are distinct points in $\D \backslash \{0\}$. 
    Then $B'$ has exactly $n-1$ zeros in $\D$.  The number of zeros in $\D_e$ is 
    $m$ if $j_0 \neq 0$ and $\leq m-1$ if $j_0 = 0$.
\end{Theorem}

\begin{proof}
    First suppose that the zeros of $B$ are distinct and that neither $B$ nor $B'$ have any zeros at the origin. 
    By \eqref{E:estb'1-0-fin}, $B'(z) =0$ if and only if
    \begin{equation*}
        \sum_{k=1}^{n}\frac{1-|z_k|^2}{(1-\overline{z_k}  z)(z-z_k)} = 0.
    \end{equation*}
    Multiplying both sides of the preceding by
    \begin{equation*}
        \prod_{k=1}^{n} (1-\overline{z_k}  z)(z-z_k),
    \end{equation*}
    we obtain a polynomial equation of degree $2n-2$ whose zeros are not in
    \begin{equation*}
        \{0,z_1,z_2,\ldots,z_n,1/\overline{z_1},1/ \overline{z_2}\ldots, 1/ \overline{z_n}\}.
    \end{equation*}
    By Lemma \ref{L:zerosindde}, there are exactly $n-1$ zeros in $\D$ and $n-1$ zeros in $\D_e$. 
    
    In the general case, Theorem \ref{T:approx-famil-distinct} permits us to
    approximate $B$ by a family $B_\epsilon$ of finite Blaschke products of order $n$ with distinct zeros 
    so that neither $B_\epsilon$ nor $B_\epsilon'$ have any zeros at the origin
    (the convergence is uniform on compact sets).  It follows that $B'$ has exactly $n-1$ zeros, 
    counted according to multiplicity, in $\D$.  However, it may 
    have fewer zeros in $\D_e$ since the zeros of $B_\epsilon'$ in $\D_e$ may be at the poles of $B$ or at $\infty$.
    
    We now consider the zeros of $B'$ in $\D_e$. First assume that $j_0 \neq 0$. Then, by direct verification, 
    \begin{equation*}
    B'(z) = z^{j_0-1}  \frac{ \prod_{k=1}^{m}(z-z_k)^{j_k-1}}{\prod_{k=1}^{m}(z-1/\overline{z_k})^{j_k+1}}  P(z),
    \end{equation*}
    in which $P$ is a polynomial of degree $2m$ with no zeros in 
    $\{0,z_1,\ldots,z_m\}$.   Consequently, $B'$ has $n+m-1$ zeros in $\C$. These are the zeros of $B$ and of $P$,
    repeated according to multiplicity.  Lemma \ref{L:zerosindde} ensures that the zeros of $P$ are of the form
    \begin{equation*}
        w_1, 1/\overline{w_1}, w_2, 1/\overline{w_2}, \ldots, w_m, 1 / \overline{w_m},
    \end{equation*}
    in which $w_1,w_2,\ldots,w_m \in \D \backslash \{0,z_1,z_2,\ldots,z_m\}$.
    
    Now suppose that $j_0 = 0$ and write 
    \begin{equation*}
        B(z) = C  \frac{ \prod_{k=1}^{m}(z-z_k)^{j_k}}{\prod_{k=1}^{m}(z-1/\overline{z_k})^{j_k}}
        = C \left( 1+ \frac{Q(z)}{\prod_{k=1}^{m}(z-1/\overline{z_k})^{j_k}} \right),
    \end{equation*}
    in which $C$ is constant and $Q$ is a polynomial of degree $n-1$.  Thus,
    \begin{equation*}
        B'(z) = \frac{ \prod_{k=1}^{m}(z-z_k)^{j_k-1}}{\prod_{k=1}^{m}(z-1/\overline{z_k})^{j_k+1}} \cdot P(z),
    \end{equation*}
    in which $P$ is a polynomial of degree at most $2m-2$ that has no zeros among $\{z_1,z_2,\ldots,z_m\}$. 
    Consequently, $B'$ has at most $n+m-2$ zeros in $\C$. These are the zeros of $B$ and of $P$,
    repeated according to multiplicity.  In this case, $P$ might have zeros at the origin. 
    For the rest of its zeros, Lemma \ref{L:zerosindde} applies. Therefore, 
    $P$ can have $\ell'$ zeros at the origin and the rest are of the form
    \begin{equation*}
        w_1,1 / \overline{w_1},  w_2, 1/ \overline{w_2}, \ldots, w_{\ell}, 1/\overline{w_{\ell}},
    \end{equation*}
    where $w_1,w_2,\ldots,w_{\ell} \in \D \backslash \{0,z_1,z_2,\ldots,z_m\}$. Since 
    $\ell'+2\ell = \deg P \leq 2m-2$, we have $\ell \leq m-1$.
\end{proof}

There are many results that relate the zeros of a polynomial to the zeros of its derivatives. 
The oldest goes back to Gauss and Lucas \cite{luc74}: the zeros of $P'$ belong to the convex hull of the zeros of $P$. In  \cite{MR1788126} they prove an analogous result in that the zeros of $B'$, the derivative of a finite Blaschke product $B$, are in the convex hull of $B^{-1}(\{0\}) \cup \{0\}$.   
A refinement of this result from \cite{MR3176999} involves some hyperbolic geometry. 
We say that $A \subseteq \D$ is \emph{hyperbolically convex} if 
\begin{equation*}
    \text{$z_1,z_2 \in A$ and $t \in [0,1]$} \quad\implies\quad
    \frac{z_1 - \frac{z_1-z_2}{1-\overline{z_1} z_2} t}{1-\overline{z_1} \frac{z_1-z_2}{1-\overline{z_1} z_2} t} \in A.
\end{equation*}
The complicated quotient in the above formula is the parametric formula for the hyperbolic segment between $z_1$ and $z_2$. 
The \emph{hyperbolic convex hull} of $A \subseteq \D$ is the smallest hyperbolic convex set that 
contains $A$; it is the intersection of all hyperbolic convex sets that contain $A$.  
Figure \ref{CH5} shows the hyperbolic convex hull of five points in $\D$.

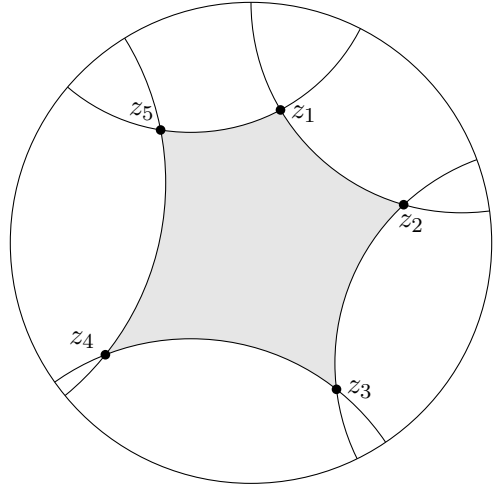
\begin{figure}
    \centering
    \begin{tikzpicture}[scale = 0.8]
        \fill[gray!20!white](0,0) circle (4);
        \begin{scope}
            \clip (0,0) circle(4);
            \pgfmathsetmacro\rone{sqrt(13)}
            \pgfmathsetmacro\rtwo{sqrt(10)}
            \pgfmathsetmacro\rthree{sqrt(21)}
            \pgfmathsetmacro\rfour{3.90512}
            \fill[white] (5,-2) circle (\rone);
            \fill[white] (3.5,4) circle (3.5);
            \fill[white] (-1,5) circle (\rtwo);
            \fill[white] (-6,1) circle (\rthree);
            \fill[white] (-1,-5.5) circle (\rfour);
            
            \draw[name path = circ1] (5,-2) circle (\rone);
            \draw[name path = circ2] (3.5,4) circle (3.5);
            \draw[name path = circ3](-1,5) circle (\rtwo);
            \draw[name path = circ4] (-6,1) circle (\rthree);
            \draw[name path = circ5] (-1,-5.5) circle (\rfour);
            
            \fill[name intersections={of=circ1 and circ2,total=\t}]
                \foreach \s in {1,...,\t}{(intersection-\s) circle (0.08) node[below,xshift=0.1cm]{$z_2$}};

            \fill[name intersections={of=circ2 and circ3,total=\t}]
                \foreach \s in {1,...,\t}{(intersection-\s) circle (0.08) node[right,yshift=-.05cm]{$z_1$}};
                
            \fill[name intersections={of=circ3 and circ4,total=\t}]
                \foreach \s in {1,...,\t}{(intersection-\s) circle (0.08) node[above,xshift=-.25cm]{$z_5$}};
                
            \fill[name intersections={of=circ4 and circ5,total=\t}]
                \foreach \s in {1,...,\t}{(intersection-\s) circle (0.08) node[left,yshift=.2cm]{$z_4$}};
                
            \fill[name intersections={of=circ1 and circ5,total=\t}]
                \foreach \s in {1,...,\t}{(intersection-\s) circle (0.08) node[right]{$z_3$}};
        \end{scope}
        \draw(0,0) circle (4);
    \end{tikzpicture}
     \caption{The hyperbolic convex hull of $z_1,z_2,\ldots,z_5$.}
        \label{CH5}
\end{figure}


\begin{Theorem} \label{T:hyper-convex}
    If $B$ is a finite Blaschke product, then the zeros of $B'$ in $\D$ belong to the hyperbolic convex hull of the zeros of $B$.
\end{Theorem}

\begin{proof}
    In what follows, we let
    \begin{equation*}
        \D_{-} = \D \cap \{ z : \Im z < 0\} \quad \text{and} \quad \D_+ = \D \cap \{ z : \Im z > 0\}.
    \end{equation*}
    First suppose that the zeros of $B$ are all in $\D_{+}$. By \eqref{E:estb'1-0-fin} we have 
    \begin{equation} \label{E:imagimarypart-bb}
        \Im \left( \frac{B'(z)}{B(z)} \right) = \sum_{k=1}^{n}
        \Im \left( \frac{1-|z_k|^2}{(1-\overline{z_k}  z)(z-z_k)} \right).
    \end{equation}
    Let
    \begin{equation*}
        \varphi(z) = \frac{1-|a|^2}{(1-\overline{a} z)(z-a)},
    \end{equation*}
    in which $a\in\D_+$ is fixed.  To study $\varphi(\D_{-})$, we examine the image of $\T_{-}$ and the interval $[-1, 1]$ under $\varphi$. On the lower semicircle
    \begin{equation*}
        \T_{-} = \{ e^{i\theta} : -\pi \leq \theta \leq 0 \},
    \end{equation*}
    we have
    \begin{equation*}
        \varphi(e^{i\theta}) = \frac{1-|a|^2}{(1-\overline{a} e^{i\theta})(e^{i\theta}-a)} = \frac{1-|a|^2}{|e^{i\theta}-a|^2} e^{-i\theta}.
    \end{equation*}
    Thus, $\T_{-}$ is mapped onto a curve in $\C_+ \cup \R$.  For $t \in [-1,1]$, we have
    \begin{equation*}
    \varphi(t) = \frac{1-|a|^2}{(1-\overline{a} t)(t-a)} = \frac{1-|a|^2}{|(1-\overline{a} t)(t-a)|^2}  (1-a t)(t-\overline{a}),
    \end{equation*}
    so
    \begin{equation*}
        \Im \varphi(t) = \frac{1-|a|^2}{(1-\overline{a} t)(t-a)} = \frac{1-|a|^2}{|(1-\overline{a} t)(t-a)|^2}  (1-t^2) \Im a.
    \end{equation*}
    Thus, $[-1,1]$ is also mapped to a curve in $\C_+ \cup \R$ and hence 
    $\partial \D_{-}$ is mapped to a closed simple curve in $\C_+ \cup \R$. 
    Since $\varphi$ is analytic on $(\D_-)^-$, we deduce that $\varphi$ maps $\D_{-}$ into $\C_+$. Equivalently, we have
    \begin{equation*}
        z \in \D_- \quad \implies \quad \Im \varphi(z) >0.
    \end{equation*}
    Since the zeros of $B$ are in $\C_+$, the representation \eqref{E:imagimarypart-bb} implies
    \begin{equation*}
        z \in \D_- \quad \implies \quad \Im \left( \frac{B'(z)}{B(z)} \right) > 0.
    \end{equation*}
    Hence, $B'$ does not have any zeros in $\D_{-}$. By continuity, it follows that if all zeros of $B$ are in $\D_{+} \cup (-1,1)$, then 
    so are the zeros of $B'$ (recall that we are only considering zeros inside $\D$).
    
    Let $f = B \circ \tau_a$. By Lemma \ref{L:balschke-conformal}, $f$ is also a finite Blaschke product with zeros 
    $\tau_a(z_1), \tau_a(z_2), \ldots, \tau_a(z_n)$.   If we denote the zeros of $B'$ in $\D$ by $w_1,w_2,\ldots,w_{n-1}$, then
    the zeros of $f'$ in $\D$ are
    \begin{equation*}
        \tau_a(w_1), \tau_a(w_2), \ldots, \tau_a(w_{n-1}).
    \end{equation*}
    If we choose $a$ such that
    $\Im \tau_a(z_k) \geq 0$ for $1 \leq k \leq n$,
    then the preceding observation shows that
    \begin{equation*}
            \qquad\Im \tau_a(w_k) \geq 0, \quad \text{for $1 \leq k \leq n-1$}.
    \end{equation*}
    This means that if the zeros of $B$ are on one side of the hyperbolic line
    \begin{equation*}
        \frac{a-z}{1-\overline{a} z} = t, \quad t \in [-1,1],
    \end{equation*}
    then the zeros of $B'$ are also on the same side. Similar comments apply if we replace $\tau_a$ by a rotation $\rho_\gamma$. 
    The intersection of all such hyperbolic lines gives the hyperbolic convex hull of the zeros of $B$.
\end{proof}

Let $a,b \in \D$ be unequal and let
\begin{equation*}
B(z) = \left( \frac{a-z}{1-\overline{a} z} \right)^m \left( \frac{b-z}{1-\overline{b} z} \right)^n.
\end{equation*}
Clearly $B'$ has $m+n-1$ zeros in $\D$ which are $a$, $m-1$ times, and $b$, $n-1$ times, and the last one $c$ which is the solution of the equation
\begin{align*}
& \frac{m(1-|a|^2)}{(1-\overline{a} z)^2} \left( \frac{a-z}{1-\overline{a} z} \right)^{m-1}  \left( \frac{b-z}{1-\overline{b} z} \right)^n\\
&\qquad+ \left( \frac{a-z}{1-\overline{a} z} \right)^m  \left( \frac{b-z}{1-\overline{b} z} \right)^{n-1} \frac{n(1-|b|^2)}{(1-\overline{b} z)^2} = 0.
\end{align*}
Rewriting this as
\begin{equation*}
\left( \frac{z-a}{1-a \overline{z}} \right)\! \Big/\! \left(\frac{z-b}{1-b \overline{z}} \right) =- \left(\frac{m(1-|a|^2)}{|1-\overline{a} z|^2} \right)\! \Big/ \!\left(\frac{n(1-|b|^2)}{|1-\overline{b} z|^2}\right)
\end{equation*}
reveals that $a,b,c$ are on the same hyperbolic line. Moreover, as $m$ and $n$ vary in $\N$, 
the point $c$ traverses a dense subset of the hyperbolic line segment between $a$ and $b$.

\section{Existence of a nonzero residue} \label{S:nonzero-residue}

The only entire finite Blaschke products are $1,z,z^2,\ldots$ and their unimodular
scalar multiples.  All other finite Blaschke products have poles in $1 < |z| < \infty$
and we can consider their residues.

\begin{Theorem}[Heins \cite{MR0043212}]  \label{T:heins}
    If $B$ is a finite Blaschke product with at least one finite
    pole, then $B$ has a nonzero residue.
\end{Theorem}

\begin{proof}
    Let
    \begin{equation} \label{E:blp} %
        B(z) = e^{i\beta} z^m \prod_{n=1}^{N} \bigg(
        \frac{z-z_n}{1-\overline{z_n}  z} \bigg)^{m_n},
    \end{equation}
    in which $z_1,z_2, \ldots, z_n$ are the distinct zeros of $B$ and let
    \begin{equation} \label{E:prim0}
        \mathfrak{B}(z)  =\int_{0}^{z} B(\zeta)\,d\zeta.
    \end{equation}
    Since $B$ is analytic on $\D^- $, the integral
    in \eqref{E:prim0} is independent of the path of integration. 
    The Fundamental Theorem of Calculus ensures that $\mathfrak{B}'(z) = B(z)$ for each $z \in \D^- $ and that $\mathfrak{B}(0)=0$. 
    By \eqref{E:blp}, for each $e^{i\theta} \in \T$ we have
    \begin{equation} \label{E:bont}
        \mathfrak{B}(e^{i\theta}) = \int_{0}^{e^{i\theta}} B(z) \,dz
        = \int_{0}^{1} e^{i\beta} r^me^{im\theta} 
        \prod_{n=1}^{N} \bigg( \frac{re^{i\theta}-z_n}{1-\overline{z_n}
        re^{i\theta}}  \bigg)^{m_n} e^{i\theta} dr.
    \end{equation}
    The function 
    \begin{equation*}
    \prod_{n=1}^{N} \bigg(\frac{re^{i\theta}-z_n}{1-\overline{z_n} re^{i\theta}} \bigg)^{m_n}
    \end{equation*}
    is a finite Blaschke product and hence \eqref{E:bont} yields
    \begin{align}
        |\mathfrak{B}(e^{i\theta})|
        &= \bigg| \int_{0}^{1} e^{i\beta} r^m e^{im\theta} 
        \prod_{n=1}^{N} \bigg( \frac{re^{i\theta}-z_n}{1-\overline{z_n} 
        re^{i\theta}}  \bigg)^{m_n}
        e^{i\theta}dr\bigg|  \nonumber \\
        &\leqslant \int_{0}^{1} r^m \,dr = \frac{1}{m+1}.  \label{sdfsdkjf} 
    \end{align}
    Perform a partial fraction expansion on \eqref{E:blp} to obtain
    \begin{equation} \label{E:bp}
        B(z) = \sum_{n=0}^{m} \alpha_n z^n + \sum_{n=1}^{N}
        \sum_{\ell=1}^{m_n} \frac{\beta_{n,\ell}}{(1-\overline{z_n}
        z)^\ell}.
    \end{equation}
    Suppose towards a contradiction that all of the residues of $B$ are zero;
    that is $\beta_{1,1} = \beta_{2,1}= \cdots =\beta_{N,1}=0$.  By integration,
    \begin{equation} \label{E:rpp}
        \mathfrak{B}(z) = \sum_{n=0}^{m} \frac{\alpha_n}{n+1} z^{n+1} +
        \alpha  + \sum_{n=1}^{N} \sum_{\ell=2}^{m_n} 
        \frac{(\frac{\beta_{n,\ell}}{\overline{z_n} (\ell-1)})}{
        (1-\overline{z_n}  z)^{\ell-1}}
    \end{equation}
    is a primitive of $B$ on $\C \backslash \{ 1/\overline{z_1},
    \ldots, 1/\overline{z_N}\}$. The constant $\alpha$ is arbitrary
    and we choose it so that $\mathfrak{B}(0)=0$. 
    
    Since $B$ has a zero of order $m$ at the origin, $\mathfrak{B}(0)=0$, and
    $\mathfrak{B}'=B$, we conclude that $\mathfrak{B}$ has a zero of order $m+1$ at the
    origin. Taking the common denominator in
    \eqref{E:rpp}, we see that
    \begin{equation} \label{E:rp}
        \mathfrak{B}(z) = \frac{z^{m+1} P(z)}{
        \prod_{n=1}^{N}(1-\overline{z_n}  z)^{m_n-1}},
    \end{equation}
    in which  $P$ is a polynomial of degree at most $\sum_{n=1}^{N}(m_n-1)$.
    On the other hand,  \eqref{E:bp} and
    \eqref{E:rpp} imply that
    \begin{equation} \label{E:flim}
        \lim_{z \to \infty} \frac{(m+1) \mathfrak{B}(z)}{z B(z)} = 1.
    \end{equation}
    Define 
    \begin{equation*}
        f(z)=\frac{(m+1) \mathfrak{B}(z)}{z  B(z)}
    \end{equation*}
    and use \eqref{E:blp} and \eqref{E:rp} to obtain 
    \begin{equation*}
        f(z) = \frac{(m+1)P(z) \prod_{n=1}^{N}(1-\overline{z_n} 
        z)}{\prod_{n=1}^{N}(z-z_n)^{m_n}},
    \end{equation*}
    which reveals that $f$ is analytic on $\C \backslash \D$.  Since $B$
    has finite poles, $f$ has at least \emph{one zero} in $\C
    \backslash \D^- $. This enables us to produce a contradiction as follows.
    By \eqref{E:flim}, we know that
    \begin{equation*}
        \lim_{z \to \infty} f(z) = 1,
    \end{equation*}
    and by \eqref{sdfsdkjf}, we have
    \begin{equation*}
        |f(e^{i\theta})| = \bigg|\frac{(m+1) 
        \mathfrak{B}(e^{i\theta})}{e^{i\theta} B(e^{i\theta})} \bigg| =
        \big| (m+1) \mathfrak{B}(e^{i\theta}) \big| \leqslant 1
    \end{equation*}
    for each $e^{i\theta} \in \T$. 
    Since $f$ is analytic on $\C \backslash \D$, $|f(\zeta)|
    \leqslant 1$ for $\zeta \in \T$, and $\lim_{z \to \infty}
    f(z) =1$, the Maximum Principle ensures that $f \equiv 1$. 
    This contradicts the assumption that $f$ has zeros
    in $\C \backslash \D$.  Thus, for some $n \in \{1,2,\ldots, N\}$,
    we must have $\beta_{n,1}\neq 0$.
\end{proof}

\section{Localization of zeros}

If $B$ is a Blaschke product of order $n$, then for each $w \in \T$ the equation $B(z)=w$ has exactly $n$ distinct solutions on $\T$;
see Corollary \ref{T:solution-n-11}.  We explore the relation between the zeros of $B$ and the solutions to $B(z)=w$ in a more general setting. We begin with an old theorem of Gauss and Lucas \cite{MR0225972}. 

\begin{Theorem}[Gauss-Lucas] \label{L:gauss-lucas}
    Let $z_1,z_2,\ldots,z_n \in \C$ be distinct and let $c_1,c_2,\ldots,c_n>0$.  Then
    \begin{equation} \label{E:def-function-f}
        f(z) = \frac{c_1}{z-z_1} +\frac{c_2}{z-z_2} + \cdots + \frac{c_n}{z-z_n}
    \end{equation}
    has $n-1$ zeros and they are in the convex hull of the set $\{z_1,z_2,\ldots,z_n\}$.
\end{Theorem}

\begin{proof}
    Multiply \eqref{E:def-function-f}
    by $(z-z_1)(z-z_2)\cdots (z-z_n)$ to see that $f$ has at most $n-1$ zeros in $\C$, counted
    according to multiplicity.  If $z_0$ is one of these zeros, then
    \begin{equation*}
        \frac{c_1}{z_0-z_1}+\frac{c_2}{z_0-z_2}+\cdots+\frac{c_n}{z_0-z_n} = 0.
    \end{equation*}
    Since $c_1,c_2,\ldots,c_n \in \R$, the preceding is equivalent to
    \begin{equation*}
        \frac{c_1(z_0-z_1)}{|z_0-z_1|^2}+\frac{c_2(z_0-z_2)}{|z_0-z_1|^2}+\cdots+\frac{c_n(z_0-z_n)}{|z_0-z_n|^2} = 0,
    \end{equation*}
    which can be written as
    \begin{equation*}
        \left( \frac{c_1}{|z_0-z_1|^2}+\cdots+\frac{c_n}{|z_0-z_n|^2} \right) z_0 =
        \frac{c_1}{|z_0-z_1|^2} z_1 +\cdots+\frac{c_n}{|z_0-z_n|^2} z_n.
    \end{equation*}
    Thus,
    \begin{equation*}
        z_0 = \lambda_1 z_1 + \lambda_2 z_2 + \cdots + \lambda_n z_n,
    \end{equation*}
    in which 
    \begin{equation*}
        \lambda_j = \frac{\frac{c_j}{|z_0-z_j|^2}}{\frac{c_1}{|z_0-z_1|^2}+\cdots+\frac{c_n}{|z_0-z_n|^2}}
        \quad \text{for $1 \leq j \leq n$}.
    \end{equation*}
    Since $c_1,c_2,\ldots,c_n > 0$, we see that
    $0<\lambda_1,\lambda_2,\ldots,\lambda_n<1$ and $\lambda_1+\cdots+\lambda_n=1$.
    Consequently, $z_0$ lies in the convex hull of the set $\{z_1,z_2,\ldots,z_n\}$.
\end{proof}

The Gauss-Lucas theorem can be used to prove a beautiful theorem about the location of the zeros of a finite Blaschke product (see \cite{MR1273115, MR1933701}). 
First we need an important lemma.

\begin{Lemma} \label{L:convex-cobm}
    Let $\{z_1,z_2,\ldots,z_{n-1}\} \subset \D$, 
    \begin{equation*}
        B(z) = z \prod_{k=1}^{n-1} \frac{z_k-z}{1-\overline{z_k}  z},
    \end{equation*}
    and $w \in \T$. Let $\zeta_1,\zeta_2,\ldots,\zeta_n$ be
        the $n$ distinct solutions to $B(\zeta)=w$ and define 
    \begin{equation*}
        \lambda_k = \frac{1}{1+\sum_{j=1}^{n-1} \frac{1-|z_j|^2}{|\zeta_k-z_j|^2}}, \quad 1 \leq k \leq n.
    \end{equation*}
    Then $\lambda_1,\lambda_2,\ldots,\lambda_n$ satisfy
    \begin{equation*}
        0<\lambda_1,\lambda_2,\ldots,\lambda_n<1 \quad \text{and} \quad \lambda_1+\lambda_2+\cdots+\lambda_n=1.
    \end{equation*}
    Moreover,
    \begin{equation*}
        \frac{B(z)/z}{B(z)-w} = \frac{(z-z_1)\cdots(z-z_{n-1})}{(z-\zeta_1)\cdots(z-\zeta_n)}= \frac{\lambda_1}{z-\zeta_1}+\cdots+\frac{\lambda_n}{z-\zeta_n}.
    \end{equation*}
\end{Lemma}

\begin{proof}
    Observe that
    \begin{equation*}
        \frac{B(z)/z}{B(z)-w} = \frac{P(z)}{Q(z)},
    \end{equation*}
    in which $P$ and $Q$ are polynomials with $\deg P = n-1$ and $\deg Q = n$.
    The roots of $P$ are $z_1,z_2,\ldots,z_{n-1}$ and the roots of $Q$ are $\zeta_1,\zeta_2,\ldots,\zeta_n$, so
    \begin{equation*}
        \frac{B(z)/z}{B(z)-w} = C\frac{(z-z_1)\cdots(z-z_{n-1})}{(z-\zeta_1)\cdots(z-\zeta_n)}
    \end{equation*}
    for some $C \neq 0$.  Multiplying through by $z$ and taking the limit as $z \to \infty$ reveals that $C = 1$.
    Performing a partial fraction expansion, we obtain
    \begin{equation*}
        \frac{B(z)/z}{B(z)-w} = \frac{\lambda_1}{z-\zeta_1}+\cdots+\frac{\lambda_n}{z-\zeta_n}.
    \end{equation*}
    Fix $1 \leq j \leq n$, multiply the previous identity by $z-\zeta_j$, then let $z \to \zeta_j$ to see that
    \begin{equation*}
        \lambda_j = \lim_{z \to \zeta_j} \frac{B(z)}{z} \cdot \frac{z-\zeta_j}{B(z)-w} = \frac{B(\zeta_j)}{\zeta_j B'(\zeta_j)}
        = \frac{1}{1+\sum_{k=1}^{n-1} \frac{1-|z_k|^2}{|\zeta_j-z_k|^2}}
    \end{equation*}
    by \eqref{E:estb'1-0-2-fin}.  Thus $0<\lambda_j<1$.  Now let $z \to \infty$ in the identity
    \begin{equation*}
        \frac{B(z)}{B(z)-w} = \frac{\lambda_1 z}{z-\zeta_1}+\cdots+\frac{\lambda_n z}{z-\zeta_n}
    \end{equation*}
    to conclude that
    \begin{equation*}
        \lambda_1+\cdots+\lambda_n = \lim_{z \to \infty} \frac{B(z)}{B(z)-w} = 1. \qedhere
    \end{equation*}
\end{proof}

\begin{Theorem} \label{T:convex-cobm}
    Let $\{z_1,z_2,\ldots,z_{n-1}\} \subset \D$, 
    \begin{equation*}
        B(z) = z \prod_{k=1}^{n-1} \frac{z_k-z}{1-\overline{z_k}  z},
    \end{equation*}
    and $w \in \T$. Let $\zeta_1,\zeta_2,\ldots,\zeta_n$ be
        the $n$ distinct solutions to $B(\zeta)=w$.
        Then $z_1,z_2,\ldots,z_{n-1}$ belong to the convex hull of $\zeta_1,\zeta_2,\ldots,\zeta_n$.
\end{Theorem}

\begin{proof}
    Lemma \ref{L:convex-cobm} yields the representation
    \begin{equation*}
        \frac{B(z)/z}{B(z)-w} = \frac{\lambda_1}{z-\zeta_1}+\cdots+\frac{\lambda_n}{z-\zeta_n},
    \end{equation*}
    in which the right-hand side is a convex combination of the functions 
    \begin{equation*}
    \frac{1}{z - \zeta_1},\quad \frac{1}{z - \zeta_2},\ldots,\quad \frac{1}{z - \zeta_n}.
    \end{equation*}
    Since the zeros of the quotient are precisely $z_1,z_2,\ldots,z_{n-1}$,
    Theorem \ref{L:gauss-lucas} ensures that they are in the convex hull of $\zeta_1,\zeta_2,\ldots,\zeta_n$.
\end{proof}

Figure \ref{CH4} illustrates Theorem \ref{T:convex-cobm} for a Blaschke product of order $n=5$.

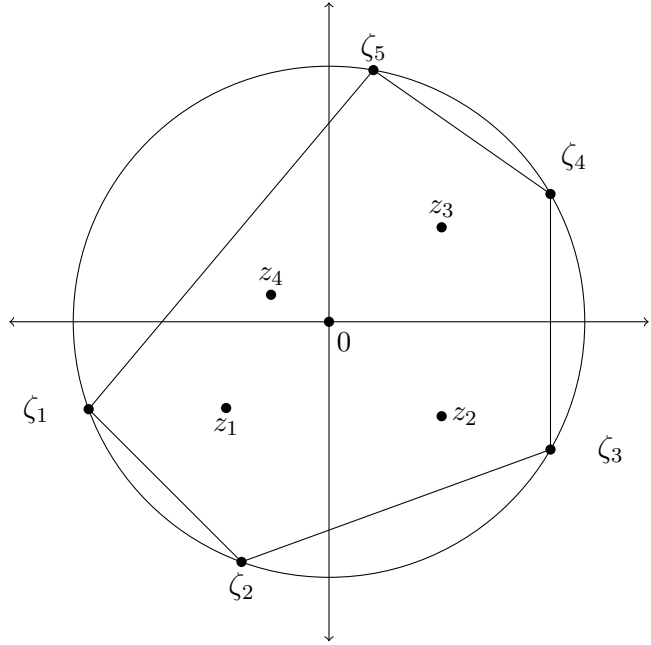
\begin{figure}
    \centering
    \begin{tikzpicture}[scale=.85]
    
        \draw(30:4) --(80:4)--(200:4)--(250:4)--(330:4)--cycle;
        \draw(0,0) circle (4);
        \fill (30:4)node[yshift=.5cm,right]{\large$\zeta_4$} circle (.08cm);
        \fill (80:4)node[yshift=.3cm]{\large$\zeta_5$} circle (.08cm);
        \fill (200:4)node[xshift=-.7cm]{\large$\zeta_1$} circle (.08cm);
        \fill (250:4)node[below]{\large$\zeta_2$} circle (.08cm);
        \fill (330:4)node[xshift=.8cm]{\large$\zeta_3$} circle (.08cm);

        \fill (40:2.3)node[above]{$z_3$} circle (.08cm);
        \fill (0,0)node[below,xshift=.2cm]{0} circle (.08cm);
        \fill (155:1)node[above]{$z_4$} circle (.08cm);
        \fill (220:2.1)node[below]{$z_1$} circle (.08cm);
        \fill (320:2.3)node[right]{$z_2$} circle (.08cm);

        \draw[<->](-5,0)--(5,0);
        \draw[<->](0,-5)--(0,5);

    \end{tikzpicture}
     \caption{The convex hull of $\zeta_1,\zeta_2,\zeta_3,\zeta_4,\zeta_5$ contain $0,z_1,z_2,z_3,z_4$.}
        \label{CH4}
\end{figure}


\begin{Corollary} \label{L:convex-cobm-version2}
    Let $\{z_1,z_2,\ldots,z_{n-1}\} \subset \D$, 
    \begin{equation*}
        B(z) = z \prod_{k=1}^{n-1} \frac{z_k-z}{1-\overline{z_k}  z},
    \end{equation*}
    and $w \in \T$. Let $\zeta_1,\zeta_2,\ldots,\zeta_n$ be
        the $n$ distinct solutions to $B(\zeta)=w$ and define 
    \begin{equation*}
        \lambda_k = \frac{1}{1+\sum_{j=1}^{n-1} \frac{1-|z_j|^2}{|\zeta_k-z_j|^2}}, \quad \text{for $1 \leq k \leq n$}.
    \end{equation*}
    Then 
    \begin{equation*}
        0<\lambda_1,\ldots,\lambda_n<1 \quad \text{and} \quad \lambda_1+\cdots+\lambda_n=1,
    \end{equation*}
    and
    \begin{equation}\label{eq:WTSTWO}
        \frac{1}{1-\overline{w} B(z)} = \frac{\lambda_1}{1-\overline{\zeta}_1 z}+\cdots+\frac{\lambda_n}{1- \overline{\zeta}_n z}.
    \end{equation}
\end{Corollary}

\begin{proof}
    Lemma \ref{L:convex-cobm} provides the convex combination
    \begin{equation*}
        \frac{B(z)/z}{B(z)-w} = \frac{\lambda_1}{z-\zeta_1}+\cdots+\frac{\lambda_n}{z-\zeta_n},
    \end{equation*}
    valid for $z \in \T$ if properly interpreted at poles.  
    For such points, $z \overline{z} = 1$ and $B(z) \overline{B(z)} = 1$, so
    \begin{equation*}
        \frac{1}{1-w \overline{B(z)}} = \frac{\lambda_1}{1-\zeta_1\overline{z}}+\cdots+\frac{\lambda_n}{1-\zeta_n\overline{z}}
    \end{equation*}
    on $\T$.  Conjugating both sides yields \eqref{eq:WTSTWO} for $z \in \T$.
    However, both sides of \eqref{eq:WTSTWO} are meromorphic on $\C$, so the identity holds everywhere.
\end{proof}

For a Blaschke product of the form
\begin{equation*}
    \qquad B(z) = z \left( \frac{\alpha-z}{1-\overline{\alpha} z} \right),\qquad \alpha\neq 0,
\end{equation*}
Theorem \ref{T:convex-cobm} has an appealing geometric interpretation. A line that passes through $\alpha$ intersects $\T$ in two points, say 
$\zeta_1$ and $\zeta_2$.   Theorem \ref{T:convex-cobm} ensures that $B(\zeta_1) = B(\zeta_2)$.
On the other hand, if $\zeta_1,\zeta_2 \in \T$ and $B(\zeta_1) = B(\zeta_2)$, then $\alpha$ lies on the line 
segment connecting $\zeta_1$ and $\zeta_2$; see Figure \ref{TP}.

\begin{figure}
    \centering
    \begin{tikzpicture}[scale=1.0]
        
        \draw (0,0) circle (3cm);
        \fill (190:3)node[left]{$\zeta_1$} circle (.04cm);
        \fill (50:3)node[right,yshift=.2cm]{$\zeta_2$} circle (.04cm);
        \fill (0,0)node[below]{0} circle (.04cm);

        \draw(190:3)--(50:3);
        \fill (.5,1.47347)node[above]{$\alpha$} circle (.04cm);

    \end{tikzpicture}
    \caption{Two points with the same image under $B$.}
    \label{TP}
\end{figure}
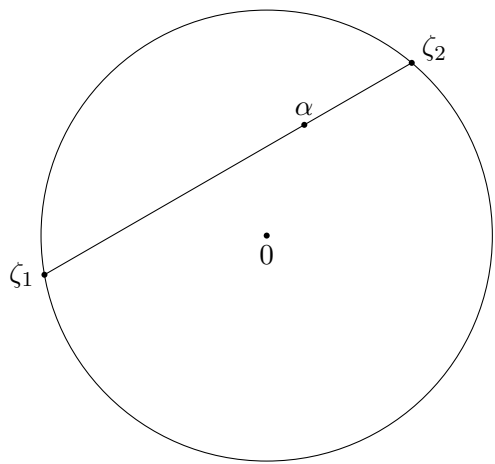


In Theorem \ref{T:convex-cobm}, we are free to choose any $w \in \T$ and 
then obtain the solutions $\zeta_{1,w},\zeta_{2,w},\ldots,\zeta_{n,w}$ of $B(\zeta)=w$. Thus, 
\begin{equation*}
    \{z_1,z_2,\ldots,z_n\} \subseteq  \bigcap_{w \in \T} \operatorname{conv}(\zeta_{1,w},\zeta_{2,w},\ldots,\zeta_{n,w}),
\end{equation*}
the intersection of the convex hulls of the sets $\{\zeta_{1,w},\zeta_{2,w},\ldots,\zeta_{n,w}\}$.
For a Blaschke product of order three, this phenomenon is depicted in Figure \ref{ELPS}. We refer the reader to \cite{MR1933701} for more on ellipses and finite Blaschke products. 

\begin{figure}
    \centering
    \begin{tikzpicture}[rotate=45, scale=0.70]
        \clip(.5,.5) circle (5.02cm);
        \draw (.5,.5) circle (5cm);

        \draw (0,0) ellipse (2cm and 4cm);
        \fill (.5,.5)node[below]{0} circle (0.08cm);
        \draw (90:4)--(-90:4);
        \fill (90:3)node[below] {$z_1$} circle (0.08cm);
        \fill (-90:3)node[below] {$z_2$} circle (0.08cm);
        
        \foreach \m in {-1.99,-1.91,...,1.99}{
            \pgfmathsetmacro{\none}{2*sqrt(4-\m*\m)}
            \pgfmathsetmacro{\ntwo}{-2*sqrt(4-\m*\m)}
            \foreach \n in {\none,\ntwo}{
                \pgfmathsetmacro{\s}{-4*\m/(\n)}
                 \draw[domain=-6:6] plot ({\x},{\s*(\x-\m)+\n});
           }
        }
    \end{tikzpicture}
    \caption{The intersection of all $\mbox{co}(\zeta_{1,w},\zeta_{2,w},\zeta_{3,w})$.}
    \label{ELPS}
\end{figure}
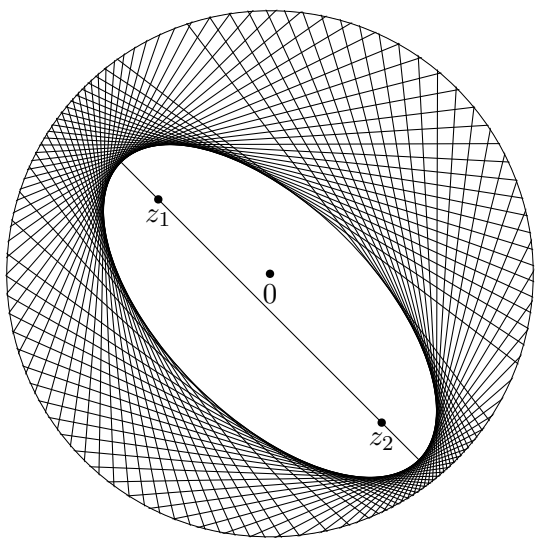


\section{The group of invariants}

For each $w \in \T$, Corollary \ref{T:solution-n-11} says that the equation 
$B(z) = w$ has exactly $n$ distinct solutions on $\T$. Thus, the sets
$B^{-1}(\{w\})$ for $w \in \T$
form a partition of $\T$; each such set has exactly $n$ elements.
In particular, consider
\begin{equation*}
    B^{-1}(1) = \{ e^{i\vartheta_1}, e^{i\vartheta_2}, \ldots, e^{i\vartheta_n}\},
\end{equation*}
in which the arguments are arranged so that
\begin{equation*}
    0 \leq \vartheta_1 < \vartheta_2 < \cdots <\vartheta_n < 2\pi.
\end{equation*}
We define $\vartheta_k$ for $k \in \Z$ by 
\begin{equation*}
    \vartheta_k = \vartheta_\ell  \pmod{2 \pi} \quad\iff\quad  k=\ell \pmod{n}.
\end{equation*}

As $\zeta$ moves once counterclockwise on $\T$, the image $B(\zeta)$ traverses
the unit circle $n$ times. This follows from the argument principle, or more explicitly, from 
Theorem \ref{L:angleont}.  As $\zeta$ passes from $e^{i\vartheta_k}$ to $e^{i\vartheta_{k+1}}$, 
the image $B(\zeta)$ makes a complete traversal of $\T$. Thus, $B$ bijectively maps  each of the arcs
\begin{equation*}
[e^{i\vartheta_1}, e^{i\vartheta_2}), \quad
[e^{i\vartheta_2}, e^{i\vartheta_3}),  \ldots, \quad
[e^{i\vartheta_n}, e^{i\vartheta_{n+1}}),
\end{equation*}
onto $\T$. If we define the equivalence relation
\begin{equation*}
e^{i\theta} \sim e^{i\theta'} \quad\iff\quad B(e^{i\theta}) = B(e^{i\theta'}),
\end{equation*}
then
\begin{equation} \label{E:t-conjug-b}
\{B^{-1}(w) : w \in \T \}
\end{equation}
is the family of equivalence classes of $\sim$; each of the arcs above contains exactly one element from each
equivalence class. 

For each $k \in \Z$, we define the bijective continuous function
\begin{equation}\label{weoirujdkn}
    \Phi_k : [e^{i\vartheta_k}, e^{i\vartheta_{k+1}})  \to \T, \qquad \Phi_{k}(e^{i\theta}) = B(e^{i\theta}).
\end{equation}
However, there are only $n$ distinct functions that arise since we have
\begin{equation*}
    \Phi_k = \Phi_\ell  \quad\iff\quad k=\ell \quad \pmod{n}.
\end{equation*}
According to the definition of the $\vartheta_k$, we see that
\begin{equation} \label{E:limiphiatends}
    \lim_{\substack{\theta \to \vartheta_k\\ \theta>\vartheta_k}} \Phi(e^{i\theta}) 
    = \lim_{\substack{\theta \to \vartheta_{k+1}\\ \theta<\vartheta_{k+1}}} \Phi(e^{i\theta}) = 1
\end{equation}
and
\begin{equation} \label{E:limiphiatends-2}
B \big( \Phi_k^{-1}(e^{i\theta})\big) = e^{i\theta}
\end{equation}
for all $e^{i\theta} \in \T$.

The following result shows that elements of the conjugacy classes \eqref{E:t-conjug-b} have a minimal lower distance from each other.

\begin{Lemma} \label{L:b-fint-distinct}
    If $B$ is a finite Blaschke product, then there is $\delta>0$ so that
    \begin{equation*}
       0 <  |e^{i\theta}-e^{i\vartheta}| < \delta \quad\implies\quad B(e^{i\theta}) \neq B(e^{i\vartheta}).
    \end{equation*}
\end{Lemma}

\begin{proof}
    Suppose towards a contradiction that for each 
    $n \geq 1$, there are points $e^{i\theta_n}, e^{i\vartheta_n} \in \T$ such that 
    \begin{equation*}
        |e^{i\theta_n}-e^{i\vartheta_n}| < \frac{1}{n}
        \quad \text{and} \quad
        B(e^{i\theta_n}) = B(e^{i\vartheta_n}).
    \end{equation*}
    By passing to a subsequence, we may assume that there is a $\theta_0$ such that
    \begin{equation*}
        \lim_{n \to \infty} e^{i\theta_n} = \lim_{n\to\infty} e^{i\vartheta_n} = e^{i\theta_0}.
    \end{equation*}
    However, Lemma \ref{L:nozerob'ont} ensures that $B$ is injective on 
    a small neighborhood of $e^{i\theta_0}$. This is a contradiction.
\end{proof}

Let $\mathcal{C}$ be the space of all continuous functions $u: \T \to \T$. This space, 
when endowed with the binary operation of function composition, is a semigroup. That is,
\begin{enumerate}
    \item $u_1, u_2 \in \mathcal{C} \implies u_1 \circ u_2 \in \mathcal{C}$;
    
    \item $(u_1 \circ u_2) \circ u_3 = u_1 \circ (u_2 \circ u_3)$ for all $u_1, u_2, u_3 \in \mathcal{C}$;
    
    \item $id \in \mathcal{C}$ (and hence $u \circ id = id \circ u = u$ for each $u \in \mathcal{C}$).
\end{enumerate}
However, an arbitrary element of $\mathcal{C}$ need not be invertible under function composition. For example, take $u(z)=z^2$.

If $B$ is a finite Blaschke product, then we may regard it as an element of $\mathcal{C}$ and define
\begin{equation*}
    G_B = \{ u \in \mathcal{C} : B \circ u = B\}.
\end{equation*}
This is a sub-semigroup of $\mathcal{C}$. In fact, much more is true.

\begin{Theorem}[Cassier--Chalendar \cite{MR1788126}] \label{T:isabelle-1}
Let $B$ be a finite Blaschke product of order $n$. Then $G_B$ is a cyclic group of order $n$.
\end{Theorem}

\begin{proof}
Consider the functions $\Phi_k$ defined in \eqref{weoirujdkn}. 
For $k \in \Z$, define functions $u_k : \T \to \T$ by
\begin{equation*}
    u_k: [e^{i\vartheta_j}, e^{i\vartheta_{j+1}}) \to [e^{i\vartheta_{j+k}}, e^{i\vartheta_{j+k+1}}), 
    \qquad u_{k}( e^{i\theta}) =  \Phi_{j+k}^{-1} \big(\Phi_j(e^{i\theta}) \big),
\end{equation*}
for $j \in \Z$. The function $u_k$, with its domain and range as restricted above, 
is a continuous bijection between two conjugacy class.  Moreover, by \eqref{E:limiphiatends},
\begin{equation*}
    u_k(e^{i\vartheta_j}) = e^{i\vartheta_{j+k}}
\quad \text{and}\quad
    \lim_{\substack{\theta \to \vartheta_{j+1}\\ \theta<\vartheta_{j+1}}} u_k(e^{i\theta}) = e^{i\vartheta_{j+k+1}}.
\end{equation*}
Upon gluing these pieces together (for fixed $k$), we obtain a continuous bijection from $\T$ onto $\T$. 
Now observe that \eqref{E:limiphiatends-2} ensures that 
\begin{equation*}
    B \big(u_k(e^{i\theta})\big) = B \big(\Phi_{j+k}^{-1} (\Phi_j(e^{i\theta}))\big) = \Phi_j(e^{i\theta}) = B(e^{i\theta})
\end{equation*}
for each $e^{i\theta} \in \T$. Thus, we obtain $n$ elements of $G_B$.

To further clarify the situation, let us make the following observations:
\begin{enumerate}
    \item $u_0 = id$;
    \item $u_k = u_\ell  \quad\iff\quad k=\ell \pmod{n}$;
    \item $u_k = u_1 \circ u_1 \circ \cdots u_1$ ($k$ times);
    \item $u_k \circ u_\ell = u_{k+\ell}$;
    \item na\"ively speaking, we may say that $u_k$ shifts forward each of the arcs
    \begin{equation*}
        [e^{i\vartheta_1}, e^{i\vartheta_2}), \quad
        [e^{i\vartheta_2}, e^{i\vartheta_3}), \ldots, \quad
        [e^{i\vartheta_n}, e^{i\vartheta_{n+1}}),
    \end{equation*}
    by $k$ steps in such a way that it preserves the equivalence classes of $\sim$. Thus, $u_k(\zeta)=\zeta'$ 
    implies that $\zeta$ and $\zeta'$ belong to the same equivalence class of $\sim$.
\end{enumerate}

The observations above reveal that $\{u_0,u_1,\ldots,u_{n-1}\}$ is a cyclic subgroup of order $n$ in $G_B$. 
We claim that this exhausts $G_B$. This fact is based on the following property: 
if $u,v \in G_B$ are such that $u(e^{i\theta_0}) = v(e^{i\theta_0})$ for some $e^{i\theta_0} \in \T$, then $u=v$. To verify this, let
\begin{equation*}
    E = \{ e^{i\theta} \in \T : u(e^{i\theta}) = v(e^{i\theta}) \}.
\end{equation*}
Clearly $e^{i\theta_0} \in E$. Since $u$ and $v$ are continuous functions, $E$ is a closed subset of $\T$. 
By uniform continuity there is a $\delta'>0$ such that
\begin{equation*}
    |e^{i\theta}-e^{i\theta'}| < \delta' \quad\implies\quad |u(e^{i\theta}) - v(e^{i\theta})| <\delta,
\end{equation*}
where $\delta>0$ is the parameter introduced in Lemma \ref{L:b-fint-distinct}. According to the definition of $G_B$ we have
\begin{equation*}
    B(u(e^{i\theta})) = B(v(e^{i\theta})) = B(e^{i\theta}), \quad e^{i\theta} \in \T.
\end{equation*}
By Lemma \ref{L:b-fint-distinct}, we must have $u(e^{i\theta}) = v(e^{i\theta})$ at least for all $e^{i\theta}$ such that $|e^{i\theta}-e^{i\theta'}| < \delta'$. This reveals that $E$ is also an open set and so $E = \T$.

To finish the proof, let $u \in G_B$. Then
\begin{equation*}
    B(u(e^{i\vartheta_1})) = B(e^{i\vartheta_1}) = 1,
\end{equation*}
and hence
\begin{equation*}
    u(e^{i\vartheta_1}) \in B^{-1}(1) = \{ e^{i\vartheta_1}, e^{i\vartheta_2}, \ldots, e^{i\vartheta_n}\}.
\end{equation*}
Suppose that $u(e^{i\vartheta_1}) = e^{i\vartheta_k}$ for some $1 \leq k \leq n$. 
If we rewrite this identity as $u(e^{i\vartheta_1}) = u_k(e^{i\vartheta_1})$, then the preceding observation shows that $u = u_k$.
\end{proof}

The following fact was stated and verified in the proof of Theorem \ref{T:isabelle-1}.

\begin{Corollary} 
Let $B$ be a finite Blaschke product. Let $u : \T \to \T$ be a continuous function such that $B \circ u = B$. Suppose that there is an $e^{i\theta_0} \in \T$ so that $u(e^{i\theta_0}) = e^{i\theta_0}$. Then $u = id$.
\end{Corollary}

\bibliography{FBP_Survey}

\end{document}